\documentclass[12pt,a4paper]{amsart}
\usepackage[centertags]{amsmath}
\usepackage{amsfonts}
\usepackage{amssymb}
\usepackage{amsthm}
\usepackage{newlfont}

\newtheorem{theorem}{Theorem}[section]
\newtheorem{proposition}{Proposition}[section]
\newtheorem{lemma}{Lemma}[section]
\newtheorem{corollary}{Corollary}[section]
\newcommand{\Tr}{\operatorname{Tr}}
\def\diag{\operatorname{diag}}
\theoremstyle{definition}

\theoremstyle{remark}
\newtheorem{remark}{Remark}

\newcommand{\R}{{\mathbf{R}}}

\newcommand{\N}{{\mathbf{N}}}

\begin{document}

\title[On generalized Powers-St$\o$rmer's Inequality]{On generalized Powers-St$\o$rmer's Inequality}

\author[D. T. Hoa]{Dinh Trung Hoa}
\address{Research Center for Sciences and Technology, Duy Tan University, 182 Nguyen Van Linh, Danang, Vietnam}
\email{dinhtrunghoa@duytan.edu.vn}

\author[H. Osaka]{HIROYUKI OSAKA$^a$} 
\date{29, Mar., 2012}
\thanks{}
\address{Department of Mathematical Sciences, Ritsumeikan
University, Kusatsu, Shiga 525-8577, Japan}
\email{osaka@se.ritsumei.ac.jp}

\author[H. M. Toan]{Ho Minh Toan}
\address{Mathematical Institute, 18 Hoang Quoc Viet, Hanoi, Vietnam}
\email{hmtoan@math.ac.vn}

%\subjclass[2000]{}


\footnote{$^A$Research partially supported by the JSPS grant for 
Scientic Research No. 20540220.}

\begin{abstract}
A generalization of Powers-St$\o$rmer's inequality for operator
monotone functions on $[0, +\infty)$ and for positive
linear functional on general $C^*$-algebras will be proved. It
also will be shown that the generalized Powers-St$\o$rmer
inequality characterizes the
tracial functionals on $C^*$-algebras.
\end{abstract}

\maketitle
\section{Introduction}
Powers-St$\o$rmer's inequality (see, for example,
\cite{Power-Stormer}) asserts that for $s \in [0,1]$ the following
inequality
\begin{equation}
2 \Tr (A^sB^{1-s}) \ge \Tr(A+B-|A-B|)
\end{equation}
holds for any pair of positive matrices $A, B$. This is a key
inequality to prove the upper bound of Chernoff bound, in quantum
hypothesis testing theory \cite{K. M. R. Audenaert}. This
inequality was first proven in \cite{K. M. R. Audenaert}, using an
integral representation of the function $t^s$. After that, M.~Ozawa
gave a much simpler proof for the same inequality, using fact that
for $s \in [0,1]$ function $f(t)=t^s\ (t\in [0,+\infty))$ is an
operator monotone. Recently, Y.~Ogata in \cite{Ogata Y} extended
this inequality to standard von Neumann algebras. The motivation
of this paper is that if the function $f(t)=t^s$ is replaced by
another operator monotone function (this class is intensively
studied, see \cite{HJT}\cite{OST}), 
then $\Tr(A+B-|A-B|)$ may get smaller upper bound that is
used in quantum hypothesis testing. Based on M.~Ozawa's proof we
formulate Powers-St$\o$rmer's inequality for an arbitrary operator
monotone function on $[0,+\infty)$ in the context of general
$C^*$-algebras.

Finally, we will show that the Powers-St$\o$rmer's inequality
characterizes the trace property for a normal linear positive functional on a von
Neumann algebras and for a  linear positive functional on a
$C^*$-algebra.

Recall that a positive linear functional $\varphi$ on a
von Neumann algebra $\mathcal{M}$ is said to be \emph{normal} if
$\varphi(\sup A_i) = \sup \varphi(A_i)$ for every bounded
increasing net $A_i$ of positive elements in $\mathcal{M}$. 
A linear functional $\varphi$ on a $C^*$-algebra $\mathcal{A}$ is
said to be \emph{tracial} if $\varphi(AB) = \varphi(BA)$ for all $A, B
\in \mathcal{A}$.

For all other notions used in the paper, we refer the reader to
the monograph \cite{Kad-Ring}.

This article has been completed when the first author visited Ritsumeikan University 
in Februrary, 2012. He is very grateful to all staffs in the department of Mathematical Sciences 
for their warm hospitality during his stay there. The second author would like to thank Professor Jun Tomiyama 
for his stimulating discussion on matrix monotone functions through e-mail.

%%%%%%%%%%%%%%%%%%%%%%%%%%%%%%%%%%%%%%%%%%%%%%%%%%%%%%%%
\section{Main results}

Let $n \in \N$ and $M_n$ be the algebra of $n \times n$ matrices.
Let $I$ be an interval in $\R$ and $f \colon I \rightarrow \R$ be a continuous function. 
We call a function $f$ matrix monotone of order $n$ or 
$n$-monotone in short whenever the inequality 
$$
A \leq B \Longrightarrow f(A) \leq f(B)
 $$
for an arbitrary selfadjoint matrices $A, B \in M_n$ such that 
$A \leq B$ and  all eigenvalues of $A$ and $B$ are contained in $I$.

Let $H$ be a separable infinite dimensional Hilbert space and $B(H)$ 
be the set of all bounded linear operators on $H$. 
We call a function $f$ operator monotone whenever 
the inequality 
$$
A \leq B \Longrightarrow f(A) \leq f(B)
 $$
for an arbitrary selfadjoint matrices $A, B \in B(H)$ such that 
$A \leq B$ and  all eigenvalues of $A$ and $B$ are contained in $I$.

\vskip 3mm

We denote the spaces of operator monotone functions by $P_\infty(I)$.  
The spaces for $n$-monotone functions are written as $P_n(I)$. 
We have then
\begin{align*}
&P_1(I) \supseteq \cdots \supseteq P_{n-1}(I) \supseteq P_n(I) \supseteq P_{n+1}(I) \supseteq \cdots 
\supseteq P_\infty(I). \\
\end{align*}
Here we note that $\cap_{n=1}^\infty P_n(I) = P_\infty(I)$ and each inclusion is proper \cite{HJT}\cite{OST}.

\vskip 3mm

The following result is well-known. For example see the proof in \cite[Theorem~2.5]{Hansen-Perd-annal}.

\vskip 3mm

\begin{lemma}\label{lem:MonotoneConcave}
Let $f$ be a strictly positive, continuous function on $[0, \infty)$. 
%\begin{enumerate}
%\item 
If the function $f$ is $2n$-monotone, then for any positive semidefinite $A$ and 
a contraction $C$ in $M_n$ we have
$$
C^*f(A)C \leq f(C^*AC).
$$
%\item
%If $f$ satisfies the inequality in $(1)$, $f$ is $n$-monotone.
%\end{enumerate}
\end{lemma}

\vskip 3mm

\vskip 3mm

\begin{lemma}\label{lem:Monotone1}
Let $f$ be a continuous function on $(0, \infty)$ such that $0 \notin f((0, \infty))$.
Then, $f$ is $n$-monotone if and only if the function $- \frac{1}{f(t)}$ is $n$-monotone.
\end{lemma}

\vskip 3mm

\begin{proof}
For any $t_1, t_2, \cdots, t_n \in (0, \infty)$ we have

\begin{align*}
\frac{\frac{1}{f(t_i)} - \frac{1}{f(t_j)}}{t_i - t_j}
&= \frac{\frac{f(t_j) - f(t_i)}{f(t_i)f(t_j)}}{t_i - t_j}\\
&= -  \frac{1}{f(t_i)f(t_j)}\frac{f(t_i) - f(t_j)}{t_i-t_j}.
\end{align*}

Since $f$ is $n$-monotone, the matrix $[\frac{f(t_i) - f(t_j)}{t_i-t_j}]$ is positive semidefinite 
by \cite{L}, hence, we have  
\begin{align*}
[\frac{(-\frac{1}{f(t_i)}) - (-\frac{1}{f(t_j)})}{t_i - t_j}]
&= - [\frac{\frac{1}{f(t_i)} - \frac{1}{f(t_j)}}{t_i - t_j}]\\
&= - (-[ \frac{1}{f(t_i)f(t_j)}\frac{f(t_i) - f(t_j)}{t_i-t_j}]) \\
&=  [\frac{1}{f(t_i)f(t_j)}] \circ [\frac{f(t_i) - f(t_j)}{t_i-t_j}] \\
&\geq 0,
\end{align*}
where $\circ$ means the Hadamard product. 

Therefore, the function $-\frac{1}{f(t)}$ is $n$-monotone by \cite{L}.

Conversely, if $-\frac{1}{f}$ is $n$-monotone, we have 
\begin{align*}
[\frac{f(t_i) - f(t_j)}{t_i-t_j}] &= [f(t_i)f(t_j)] \circ 
[\frac{(-\frac{1}{f(t_i)}) - (-\frac{1}{f(t_j)})}{t_i - t_j}]\\
&\geq 0,
\end{align*}
hence $f$ is $n$-monotone.
\end{proof}

\vskip 3mm

\begin{proposition}\label{prp:monotone}
Let $f$ be a strictly positive, continuous function on $[0, \infty)$. 
If $f$ is $2n$-monotone, the function $g(t) = \frac{t}{f(t)}$ is $n$-monotone 
on $[0, \infty)$.
\end{proposition}

\vskip 3mm

\begin{proof}
Let $A, B$ be positive matrixces in $M_n$ such that $0 < A \leq B$.

Let $C = B^{-\frac{1}{2}}A^{\frac{1}{2}}$. Then $\|C\| \leq 1$.
Since $f$ is $2n$-monotone, $-f$ satisfies the Jensen type inequality  from Lemma~\ref{lem:MonotoneConcave},  
that is, 
\begin{align*}
-f(A) = -f(C^*BC) &\leq -C^*f(B)C\\
-f(A) &\leq -A^{\frac{1}{2}}B^{-\frac{1}{2}}f(B)B^{-\frac{1}{2}}A^{\frac{1}{2}}\\
-A^{-\frac{1}{2}}f(A)A^{-\frac{1}{2}} &\leq - B^{-\frac{1}{2}}f(B)B^{-\frac{1}{2}}\\
-A^{-1}f(A) &\leq -B^{-1}f(B)
\end{align*}
Hence, the function $-\frac{f(t)}{t}$ is $n$-monotone.
Therefore, from Lemm~\ref{lem:Monotone1} we conclude that 
$$
-\frac{1}{-\frac{f(t)}{t}} = \frac{t}{f(t)}
$$
is $n$-monotone.
\end{proof}

\vskip 3mm

\begin{remark}
The condition of $2n$-monotonicity of $f$ is needed to guarantee the $n$-monotonicity of  
$g$. Indeed, 
it is well-known that $t^3$ is monotone, but not $2$-monotone. 
In this case the function $g(t) = \frac{t}{t^3} = \frac{1}{t^2}$ 
is obviously not $1$-monotone.
\end{remark}

\vskip 3mm

\begin{proposition}\label{prp:monotone2}
Let $h \colon [0, \infty) \rightarrow [0, \infty)$ be a Borel function such that 
$h$ is a continuous, $n$-monotone on $(0, \infty)$, and $h(0) = 0$.
Then for any $A, B \in M_n^+$ with $A \leq B$ we have 
$$
h(A) \leq h(B).
$$
\end{proposition}

\begin{proof}
Let $B = \sum_s \mu_s q_s$ be a spectral decomposition. 
Set $1 - q$ as a projection on $\mathrm{Ker}(B)$. Then $B = Bq = qB = \sum_{s'}\mu_{s'}q_{s'}$ and 
$q = \sum_{s'} q_{s'}$. 

Similarly, let $A = \sum_t \lambda_tp_t$ be a spectral projection and $(1 - p)$ be a 
projection on $\mathrm{Ker}(A)$. Since $A \leq B$, $p \leq q$ and 
$A = \sum_{t'}\lambda_{t'}p_{t'}$ and $p = \sum_{t'}p_{t'}$. 
Note that  since $h(0) = 0$, by the function calculus we have 
$h(A) = \sum_{t'}h(\lambda_{t'})p_{t'}$ and $h(B) = \sum_{s'}h(\mu_{s'})q_{s'}$.

For any $\varepsilon > 0$ since
\begin{align*}
0 &< \sum_{t'}\lambda_{t'}p_{t'} + \varepsilon 1\\
&\leq \sum_{s'}\mu_{s'}q_{s'} + \varepsilon 1
\end{align*}
and $h$ is $n$-monotone on $(0, \infty)$, 
we have
\begin{align*}
h(\sum_{t'}(\lambda_{t'} + \varepsilon)p_{t'} + \varepsilon (1 - p)) \leq
h(\sum_{s'}(\mu_{s'} + \varepsilon) q_{s'} + \varepsilon (1 - q)).
\end{align*}

Since 
\begin{align*}
 \sum_{t'}h(\lambda_{t'}+ \varepsilon)p_{t'} + h(\varepsilon) (1 - p)
&= h(\sum_{t'}(\lambda_{t'}  + \varepsilon) p_{t'} + \varepsilon (1 - p)) \\
& \leq h(\sum_{s'}(\mu_{s'} + \varepsilon) q_{s'} + \varepsilon (1 - q)) \\
&= \sum_{s'}h(\mu_{s'} + \varepsilon)q_{s'} + h(\varepsilon) (1 - q)\\
\end{align*}
and $p \leq q$, 
it follows that 
\begin{align*}
\sum_{t'}h(\lambda_{t'}+ \varepsilon)p_{t'} 
&\leq \sum_{t'}h(\lambda_{t'}+ \varepsilon)p_{t'} + h(\varepsilon)q(1 - p)q\\
&\leq \sum_{s'}h(\mu_{s'} + \varepsilon)q_{s'}.
\end{align*}
Therefore, since $h$ is continuus on $(0, \infty)$, as $\varepsilon \rightarrow 0$ we have 
\begin{align*}
h(A) &= \sum_{t'}h(\lambda_{t'})p_{t'} \\
&\leq \sum_{s'}h(\mu_{s'})q_{s'}\\
&= h(B).
\end{align*}
\end{proof}

\vskip 3mm

\begin{corollary}\label{coro:Borel function calculus}
Let $f$ be a $2n$-monotone, continuous  
function on $[0, \infty)$ such that $f((0, \infty)) \subset (0, \infty)$, 
and let $g$ be a Borel function on $[0, \infty)$ defined by 
$g(t) 
= 
\left\{\begin{array}{cc}
\frac{t}{f(t)}&(t \in (0, \infty))\\
0& (t = 0)
\end{array}
\right.
$.
Then for any pair of positive matrices $A, B \in M_n$ with $A \leq B$, 
$g(A) \leq g(B)$.
\end{corollary}

\vskip 3mm

\begin{proof}
Since $f$ is $2n$-monotone, continuous function on $[0, \infty)$ such that $f((0, \infty)) \subset (0, \infty)$, 
from Proposition~\ref{prp:monotone} $g$ is $n$-monotone on $(0, \infty)$. 

Hence, since $g$ is a Borel function on $[0, \infty)$ with $g(0) = 0$, from Proposition~\ref{prp:monotone2} it follows that 
$g(A) \leq (B)$.
\end{proof}

\vskip 3mm

\begin{theorem}\label{thm:Powers-Stormer}
Let $\Tr$ be a canonical trace on $M_n$ and $f$ be a $2n$-monotone 
function on $[0, \infty)$ such that $f((0, \infty)) \subset (0, \infty)$. 
Then for any pair of positive matrices $A, B \in M_n$  

\begin{equation}\label{matrix inequality}
\Tr(A) + \Tr(B) - \Tr(|A - B|) 
\leq 2\Tr(f(A)^{\frac{1}{2}}g(B)f(A)^{\frac{1}{2}}),
\end{equation}
where 
$g(t) 
= 
\left\{\begin{array}{cc}
\frac{t}{f(t)}&(t \in (0, \infty))\\
0& (t = 0)
\end{array}
\right.
$.
\end{theorem}

\vskip 3mm

%\begin{theorem}\label{thm:Powers-Stormer}
%Let $\Tr$ be a canonical trace on $M_n$ and $f$ be a strictly positive, $2n$-monotone 
%function on $[0, \infty)$. Then for any pair of positive matrices $A, B \in M_n$  

%\begin{equation}\label{matrix inequality}
%\Tr(A) + \Tr(B) - \Tr(|A - B|) \leq 2\Tr(f(A)^{\frac{1}{2}}g(B)f(A)^{\frac{1}{2}}),
%\end{equation}
%where $g(t) = \frac{t}{f(t)}$.
%\end{theorem}

\vskip 3mm

\begin{proof}
Let $A, B$ be any positive matrices in $M_n$.

For operator $(A - B)$ let us denote by $P = (A - B)^+$ and $Q =(A - B)^-$ 
its positive and negative part, respectively. Then we have
\begin{equation}\label{condition 1}
A - B = P - Q \quad \text{and} \quad |A - B| = P + Q,
\end{equation}
from that it follows that
\begin{equation}\label{condition 2}
A + Q = B + P.
\end{equation}

On account of (\ref{condition 2}) the inequality (\ref{matrix inequality}) is equivalent to the following
\begin{equation*}
\Tr(A) - \Tr (f(A)^{\frac{1}{2}} g(B) f(A)^{\frac{1}{2}}) \leq \Tr(P).
\end{equation*}

Since $B + P \ge B \ge 0$ and $B + P = A + Q \ge A \ge 0$,  we have 
$g(A) \leq g(B + P)$ by Corollary~\ref{coro:Borel function calculus} and 

\begin{align*}
\Tr(A)& - \Tr(f(A)^{\frac{1}{2}}g(B) f(A)^{\frac{1}{2}})\\
&= \Tr(f(A)^{\frac{1}{2}}g(A) f(A)^{\frac{1}{2}}) -  \Tr(f(A)^{\frac{1}{2}}g(B) f(A)^{\frac{1}{2}})\\
& \le \Tr(f(A)^{\frac{1}{2}}g(B + P) f(A)^{\frac{1}{2}}) - \Tr(f(A)^{\frac{1}{2}} g(B) f(A)^{\frac{1}{2}})\\
& = \Tr(f(A)^{\frac{1}{2}} (g(B + P) - g(B)) f(A)^{\frac{1}{2}}) \\
& \le \Tr(f(B + P)^{\frac{1}{2}} (g(B + P) - g(B)) f(B + P)^{\frac{1}{2}}) \\
& = \Tr(f(B + P)^{\frac{1}{2}} g(B + P)f(B + P)^{\frac{1}{2}}) \\
& - \Tr(f(B + P)^{\frac{1}{2}} g(B)f(B + P)^{\frac{1}{2}}) \\
& \leq \Tr(B + P) - \Tr(f(B)^{\frac{1}{2}} g(B)f(B)^{\frac{1}{2}}) \\
& = \Tr(B + P) - \Tr(B)\\
& = \Tr(P).
\end{align*}

Hence, we have the conclusion.

\end{proof}

\vskip 3mm

\begin{remark}
\begin{enumerate}
\item[(i)]
When given positive matrices $A, B$ in $M_n$ satisfies the condition $A \leq B$,
the inequality (\ref{matrix inequality}) becomes
\begin{align*}
\Tr(A) \leq \Tr(f(A)^{\frac{1}{2}}g(B)f(A)^{\frac{1}{2}}).
\end{align*}
\item[(ii)]
As pointed in Proposition~\ref{prp:monotone}, $2$-monotonicity of $f$ is needed to guarantee the 
inequality (\ref{matrix inequality}). Indeed, let $f(t) = t^3$ and $n = 1$.
Then, for any
$a, b \in (0, \infty)$,  the inequality (\ref{matrix inequality}) would imply 
\begin{align*}
a \leq f(a)^{\frac{1}{2}}g(b)f(a)^{\frac{1}{2}},
\end{align*}
that is,
\begin{align*}
\frac{a}{f(a)} \leq \frac{b}{f(b)}. 
\end{align*}
Since $\frac{t}{f(t)}$ is, however, not $1$-monotone, the latter inequality is impossible.
\end{enumerate}
\end{remark}

\vskip 3mm

As an application we get Powers-St\o rmer's inequality.

\vskip 3mm

\begin{corollary}\cite[Theorem~1]{K. M. R. Audenaert}
Let $A$ and $B$ be positive matrices, then for all $s \in [0, 1]$ 
$$
\Tr(A + B - |A - B|) \leq \Tr(A^sB^{1-s}).
$$
\end{corollary}

\vskip 3mm

\begin{proof}
Let $f(t) = t^s$ \ $(s \in [0,1])$. Then $f$ is operator monotone with $f(0, \infty) \subset (0, \infty)$ 
and $g(t) = t^{1-s}$.
Hence, we have the conclusion from Theorem~\ref{thm:Powers-Stormer}.
\end{proof}

\vskip 3mm

Since any C*-algebra can be realized as a closed selfadjoint $*$-algebra of $B(H)$ for some Hilbert space $H$. 
We can generalize Theorem~\ref{thm:Powers-Stormer} in the framework of C*-algebras.

\vskip 3mm

\begin{theorem}\label{theorem 1}
Let $\tau$ be a tracial functional on a $C^*$-algebra
$\mathcal{A}$, $f$ be a strictly positive, operator
monotone function on $[0, \infty)$. Then for any pair of positive elements $A, B \in \mathcal{A}$ 

\begin{equation}\label{main inequality}
\tau(A) + \tau(B) - \tau(|A-B|) \le 2 \tau (f(A)^{\frac{1}{2}} g(B) f(A)^{\frac{1}{2}}),
\end{equation}
where $g(t) = tf(t)^{-1}$.
\end{theorem}

\begin{proof}
Since the function $\frac{t}{f(t)}$ is operator monotone on $(0, \infty)$ by \cite[Corollary 6]{Hansen-Perd-annal},
we can get the conclusion through the same steps in the proof of Theorem~\ref{thm:Powers-Stormer}.
\end{proof}

%\begin{proof}
%For operator $(A-B)$ let us denote by $P = (A-B)^+$ and $Q =
%(A-B)^-$ its positive and negative part, respectively. Then we
%have
%\begin{equation}\label{condition 1}
%A-B=P-Q \quad \text{and} \quad |A-B|=P+Q,
%\end{equation}
%from that it follows
%\begin{equation}\label{condition 2}
%A+Q = B+P.
%\end{equation}

%On account of (\ref{condition 2}) the inequality (\ref{main
%inequality}) is equivalent to the following
%\begin{equation*}
%\tau(A) - \tau (f(A)^{1/2} g(B) f(A)^{1/2}) \le \tau(P).
%\end{equation*}

%Note that \cite[Corollary 6]{Hansen-Perd-annal} for given $f(t)$
%function $g(t)=tf(t)^{-1}$ is an operator monotone. Since $B+P \ge
%B \ge 0$ and $B+P = A+Q \ge A \ge 0$ we have
%\begin{equation*}
%\begin{split}
% \tau(A)& - \tau(f(A)^{1/2} g(B) f(A)^{1/2})  \\
%& = \tau(f(A)^{1/2} g(A) f(A)^{1/2}) - \tau(f(A)^{1/2} g(B) f(A)^{1/2}) \\
%& \le \tau(f(A)^{1/2} (g(B+P)-g(B)) f(A)^{1/2}) \\
%& \le  \tau(f(B+P)^{1/2} (g(B+P)-g(B)) f(B+P)^{1/2}) \\
%& = \tau(f(B+P)^{1/2}g(B+P)f(B+P)^{1/2}) - \tau(f(B+P)^{1/2}g(B) f(B+P)^{1/2}) \\
%& \le \tau(B+P)  - \tau(f(B)^{1/2}g(B) f(B)^{1/2}) \\
%& = \tau(B+P)-\tau(B) \\
%& = \tau(P).
%\end{split}
%\end{equation*}
%\end{proof}

\begin{remark}
For matrices $A, B\in M_n^+$ let us denote
\begin{equation}
Q(A,B) = \min_{s\in [0,1]} \Tr(A^{(1-s)/2}B^sA^{(1-s)/2})
\end{equation}
and
\begin{equation}
Q_{\mathcal{F}_{2n}} (A, B)=\inf_{f \in \mathcal{F}_{2n}} \Tr(f(A)^{\frac{1}{2}} g(B)
f(A)^{\frac{1}{2}}),
\end{equation}
where $\mathcal{F}_{2n}$ is the set of all $2n$-monotone functions
on $[0,+\infty)$ satisfy condition of the Theorem \ref{thm:Powers-Stormer}
and $g(t)=tf(t)^{-1}\ (t \in [0,+\infty)).$

Note that the function $f(t)=t^s\ (t\in [0,+\infty))$ satisfies the
conditions of Theorem \ref{thm:Powers-Stormer}. Since the  class of 
$2n$-monotone functions is large enough \cite{OST}, we know that 
$Q_{\mathcal{F}_{2n}}(A,B) \le Q(A,B)$.
Hence, we hope on finding
another $2n$-monotone function $f$ on $[0,+\infty)$ such that
\begin{equation}
\Tr(f(A)^{\frac{1}{2}} g(B) f(A)^{\frac{1}{2}}) < Q(A,B).
\end{equation}
If we can find such a function, then we can refine the quantum
Chernoff bound used in quantum hypothesis testing \cite{K. M. R.
Audenaert}.
\end{remark}

%%%%%%%%%
%%%%%%%

%\section{Inequality in singular values}

%\section{Case of equations}
%%%%%%%%%%%%%%%%%%%%%%%%%%%%%%%%%%%%%%%%%%%%%%%%%%%%%%%%%%%%%%%%%%%%%%%%
\section{Characterizations of the trace property}

In this section the generalized Powers-St\o rmer inequality  in the previous 
section implies the trace property for a positive linear functional on 
operator algebras.

\vskip 3mm

\begin{lemma}\label{lemma 2}
Let $\varphi$ be a positive linear functional on $M_n$ and $f$ 
be a continuous function on $[0, \infty)$ such that $f(0) = 0$ and 
$f((0, \infty)) \subset (0, \infty)$.
If the following inequality
\begin{equation}\label{Chernoff bound}
\varphi(A+B) - \varphi(|A-B|) \le 2\varphi(f(A)^{\frac{1}{2}}g(B)f(A)^{\frac{1}{2}})
\end{equation}
holds true for all $A, B \in M_n^+$, then $\varphi$ should be a
positive scalar multiple of the canonical trace $\Tr$ on $M_n$, where 
$g(t) = \left\{\begin{array}{cl}
\frac{t}{f(t)}& (t \in (0, \infty))\\
0 & (t = 0)
\end{array}
\right.
$.
\end{lemma}

\begin{proof}
As is well known, every positive linear functional $\varphi$ on
$M_n$ can be represented in the form $\varphi(\cdot) =
\Tr(S_{\varphi}\cdot)$ for some $S_{\varphi} \in M_n^+.$ It is
easily seen that without loss of generality we can assume that
$S_{\varphi} = \diag (\alpha_1, \alpha_2, \cdots, \alpha_n),$ and
we have to prove that $\alpha_i = \alpha_j$ for all $i, j =1,
\cdots, n.$ Clearly, it is sufficient to prove that $\alpha_1 =
\alpha_2.$ By assumption, the inequality (\ref{Chernoff bound}) holds
true, in particular, for any positive matirices matrices
$X=[x_{ij}]_{i,j=1}^n, Y=[y_{ij}]_{i,j=1}^n$ from $M_n^+$ such that
$0 = x_{ij}=y_{ij}$ if $3\leq i\leq n$ or $3\leq j\leq n.$ 
Thus, it suffices to consider the case $n = 2.$  
Assume that
$S_\varphi=\mbox{diag}(d,1)~(d \in [0,1])$ and
$\varphi(D)=\mbox{Tr}(S_\varphi D), \forall D \in M_2.$ We show
that $d = 1.$ For arbitrary positive numbers $\lambda, \mu$ such
that $\lambda < \mu$ we consider the following matrices
$$A= \left(%
\begin{array}{cc}
  \lambda & \sqrt{\lambda \mu} \\
  \sqrt{\lambda \mu} & \mu \\
\end{array}%
\right)
$$
and
$$B=\left(%
\begin{array}{cc}
  \lambda & -\sqrt{\lambda \mu} \\
  -\sqrt{\lambda \mu} & \mu \\
\end{array}%
\right).
$$
It is clear that these are positive scalar multiple of projections
of rank one. In addition,
\begin{equation*}
f(A)^{\frac{1}{2}}g(B)f(A)^{\frac{1}{2}} = \left(\frac{\mu - \lambda}{\mu +
\lambda}\right)^2 A.
\end{equation*}
We have
\begin{equation*}
\begin{split}
 2\varphi(f(A)^{\frac{1}{2}}g(B)f(A)^{\frac{1}{2}}) & = 2 \left(\frac{\mu -
\lambda}{\mu + \lambda}\right)^2 \Tr (S_\varphi A) \\
& = 2 \left(\frac{\mu - \lambda}{\mu + \lambda}\right)^2 (d \lambda
+ \mu).\\
\end{split}
\end{equation*}
By direct calculation,
\begin{equation*}
|A-B| = \left(%
\begin{array}{cc}
  2\sqrt{\lambda \mu} & 0 \\
  0 & 2\sqrt{\lambda \mu} \\
\end{array}%
\right).
\end{equation*}

Consequently,
\begin{equation*}
\varphi(A+B) - \varphi(|A-B|) = d(2\lambda - 2 \sqrt{\lambda \mu})
+ 2\mu - 2 \sqrt{\lambda \mu}.
\end{equation*}
Then the inequality (\ref{Chernoff bound}) becomes
\begin{equation*}
\left(\frac{\mu - \lambda}{\mu + \lambda}\right)^2 (d \lambda +
\mu) \ge d(\lambda -  \sqrt{\lambda \mu}) + \mu -  \sqrt{\lambda \mu}.
\end{equation*} 

Dividing two side by $\sqrt{\lambda}
(\sqrt{\mu}-\sqrt{\lambda})$, we get
\begin{equation*}
d +
\frac{(\sqrt{\mu}-\sqrt{\lambda})(\sqrt{\mu}+\sqrt{\lambda})^2}{\sqrt{\lambda}(\mu
+ \lambda)^2} (d \lambda + \mu) \ge \sqrt{\frac{\mu}{\lambda}}.
\end{equation*}

Tending $\lambda$ to $\mu$ from the  left we obtain
$$ d \ge 1.$$

Since $d \in [0, 1]$, $d = 1$. 
This means that $\varphi$ is a positive scalar multiple of 
the canonical trace $\Tr$ on $M_n$.
\end{proof}

\begin{remark}\label{rem.4}
Let $\varphi$ be a positive linear functional on $M_n$ and $s\in [0,1]$. From Lemma \ref{lemma 2} it is clear that if the following inequality 
\begin{equation}\label{PS}
\varphi(A+B) - \varphi(|A-B|) \le 2\varphi(A^{\frac{1-s}{2}}B^s A^{\frac{1-s}{2}})
\end{equation}
holds true for any $A, B \in M_n^+$, then $\varphi$ is a tracial. In particular, when $s=0$ the following inequality characterizes the trace property
\begin{equation}\label{PS}
\varphi(B) - \varphi(A) \le \varphi(|A-B|) \quad (A, B \in M_n^+). 
\end{equation}

\end{remark} 
\begin{corollary}[\cite{Tikhonov-Sherstnev}]
Let $\varphi$ be a positive linear functional on $M_n$ and the following inequality 
\begin{equation}\label{cor1}
\varphi(|A+B|) \le \varphi(|A|) + \varphi(|B|)
\end{equation} 
holds true for any self-adjoint matrices $A,B \in M_n$. Then $\varphi$ is a tracial.
\end{corollary}
\begin{proof}
From the assumption, we have 
\begin{equation*}
\varphi(|B-A|) \ge \varphi(|B|) -\varphi(|A|)
\end{equation*} 
for any pair of self-adjoint matrices $A,B$ in $M_n$. Moreover, for any pair of positive matrices $A,B\in M_n$ we have
\begin{equation*}
\varphi(|B-A|) \ge \varphi(B) -\varphi(A).
\end{equation*}
On account of Remark \ref{rem.4}, it follows that $\varphi$ should be a tracial.
\end{proof}

\begin{corollary}[\cite{Gardner}]
Let $\varphi$ be a positive linear functional on $M_n$ and the following inequality 
\begin{equation}\label{cor2}
|\varphi(A)| \le \varphi(|A|)
\end{equation} 
holds true for any self-adjoint matrix $A \in M_n$. Then $\varphi$ is a tracial.
\end{corollary}
\begin{proof}
Let $A,B\in M_n$ be arbitrary positive matrices. Then $C=B-A$ is a self-adjoint matrix. Since $A,B \ge 0$, the  values $\varphi(A)$ and $\varphi(B)$ are real. From the assumption, we have
\begin{equation*}
\varphi(B) - \varphi(A) \le |\varphi(B) - \varphi(A)| = |\varphi(B - A)|  \le \varphi(|B-A|).
\end{equation*} 
On account of Remark \ref{rem.4}, it follows that $\varphi$ should be a tracial.
\end{proof}
\vskip 3mm

By analogy with a number of other similar cases (see \cite{Gardner} or \cite{Tikhonov}), 
the proof for the trace property of a positive normal functional satisfying 
the inequality~(\ref{Chernoff bound}) on a von Neumann algebra can be reduced to the case of the algebra
$M_2$ of all matrices of order $2\times 2$. But for self-contained 
we will give a sketch of its proof.

\vskip 3mm

\begin{theorem}\label{char.von.theorem}
Let $\varphi$ be a positive normal linear functional on a von Neumann
algebra $\mathcal{M}$ and $f$ be a continuous function on $[0, \infty)$ 
such that $f(0) = 0$ and $f((0, \infty)) \subset (0, \infty)$.  
If the following inequality
\begin{equation}\label{char.von}
\varphi(A) + \varphi(B) - \varphi(|A-B|) \le 2 \varphi (f(A)^{\frac{1}{2}}g(B) f(A)^{\frac{1}{2}})
\end{equation}
holds true for any pair $A, B \in \mathcal{M}^+$, then $\varphi$
is a trace, where 
$g(t) = \left\{\begin{array}{cl}
\frac{t}{f(t)}& (t \in (0, \infty))\\
0 & (t = 0)
\end{array}
\right.
$.
\end{theorem}

\begin{proof}
Let $P_1, P_2$ be a pair of nonzero mutually orthogonal equivalent
projections in $\mathcal{M}$, that means $P_1 = V^*V$ and $P_2 =
VV^*$ for some nonzero partial isometry $V \in \mathcal{M}$.
Consider the $*$-algebra $\mathcal{N}$ in
$(P_1+P_2)\mathcal{M}(P_1+P_2)$ generated by the partial isometry
$V$. Then $\mathcal{N}$ is isomorphic to $M_2$ and
inequality (\ref{char.von}) still holds true for the operators in
$\mathcal{N}$ and for the restriction of the functional $\varphi$
to $\mathcal{N}$. According to Lemma \ref{lemma 2}, this
restriction is a tracial functional on $\mathcal{N}$, and hence
$\varphi(P_1) = \varphi(P_2)$. 
By \cite[Vol2, Proposition~8.1.1]{Kad-Ring}
it follows that $\varphi$ is a trace.
\end{proof}

\vskip 3mm

\begin{corollary}
Let $\varphi$ be a positive linear functional on a $C^*$-algebra
$\mathcal{A}$ and $f$ be a continuous function on $[0, \infty)$
such that $f(0) = 0$ and $f((0, \infty)) \subset (0, \infty)$. 
If the following inequality
\begin{equation}\label{char.C}
\varphi(A) + \varphi(B) - \varphi(|A-B|) \le 2 \varphi (f(A)^{\frac{1}{2}}g(B) f(A)^{\frac{1}{2}})
\end{equation}
holds true for any pair $A, B \in \mathcal{A}^+$, then $\varphi$
is a tracial functional, where 
$g(t) = \left\{\begin{array}{cl}
\frac{t}{f(t)}& (t \in (0, \infty))\\
0 & (t = 0)
\end{array}
\right.
$.
\end{corollary}

\begin{proof}
Let $\pi$ be the universal representation of $C^*$-algebra
$\mathcal{A}$ and $\mathcal{M}=\pi(\mathcal{A})''$. Let
$\hat{\varphi}$ be the positive normal functional on $\mathcal{M}$
such that $\hat{\varphi}(\pi(A)) = \varphi(A)$ for $A \in
\mathcal{A}$. By the Kaplansky density theorem, for any pair
$\hat{A}, \hat{B} \in \mathcal{M}^+$ there exist bounded nets
$\{A_\alpha\}$ and $\{B_\alpha\}$ in $\mathcal{A}^+$ such that
$\pi(A_\alpha) \rightarrow \hat{A}$ and $\pi(B_\alpha) \rightarrow
\hat{B}$ in the strong operator topology. Using (\ref{char.C}) and
the continuity of the corresponding operations in the strong
operator topology, we have
\begin{equation*}
\hat{\varphi}(\hat{A}) + \hat{\varphi}(\hat{B}) -
\hat{\varphi}(|\hat{A}-\hat{B}|) \le 2 \hat{\varphi}
(f(\hat{A})^{\frac{1}{2}} g(\hat{B}) f(\hat{A})^{\frac{1}{2}}).
\end{equation*}
By Theorem \ref{char.von.theorem}, $\hat{\varphi}$ is a tracial
functional $\mathcal{M}$, and hence $\varphi$ is a tracial
functional on $\mathcal{A}$.
\end{proof}

%\textbf{Question:} What is Powers-Stormer's inequality for
%singular values?

\vskip 3mm

\begin{remark}
%When $\mathcal{A}$ is a von Neumann algebra, the set of all orthogonal projections from
%$\mathcal{A}$ is enough as the testing set for inequality~(\ref{char.von})
%for $\varphi$ to have the  trace property
%(\cite{Bikchentaev}). In the case of C*-algebras, however, this set is not enough as the testing set. 

%({\bf Need revise the above sentences.})

Let $\mathcal{A}$ be a von Neumann algebra and $\varphi$ be a positive linear functional on $\mathcal{A}$.
The set $P(\mathcal{A})$ of all orthogonal projections from $\mathcal{A}$ is enough as a testing space 
for some inequlity to characterize the trace property of $\varphi$ (see \cite{Bikchentaev}). 
But, in the case of the inequality~(\ref{char.von}) the set $P(\mathcal{A})$ is not enough as a testing set.

Indeed,
let $p, q$ be arbitrary orthogonal projections from a von Neumann
algebra $\mathcal{M}$. Since $q\geq p \wedge q$ it follows that $p
q p \geq p (p \wedge q) p= p \wedge q.$ So $pqp \ge p \wedge q$
holds for any pair of projections. From that it follows

%\begin{align*}
\begin{equation*}
\varphi(p+q-|p-q|) = 2\varphi(p \wedge q) \le 2 \varphi(pqp)
= 2\varphi(f(p)^{\frac{1}{2}}g(q)f(p)^{\frac{1}{2}})
\end{equation*}
%\end{align*}

whenever $\varphi$ is an arbitrary positive linear functional on
$\mathcal{M}$.
\end{remark}

\end{document}